\documentclass[11pt]{amsproc}

\usepackage{mathtext}
\usepackage[cp1251]{inputenc}

\usepackage{bm}

\usepackage{amsmath}
\usepackage{amssymb}
\usepackage{amsxtra}

\def\Q{{{\Bbb Q}}}
\def\N{{{\Bbb N}}}
\def\Z{{{\Bbb Z}}}
\def\T{{{\Bbb T}}}
\def\R{{\Bbb R}}
\def\C{{\Bbb C}}

\def\a{{\alpha }}
\def\D{{\Delta }}

\def\a{{\alpha}}
\def\b{{\beta}}

\def\e{{\varepsilon}}

\def\vp{{\varphi}}

\def\g{{\gamma }}

\def\L{{\Lambda }}

\def\La{{\Lambda }}

\def\){\right)}
\def\({\left(}

\numberwithin{equation}{section}
\newtheorem{corollary}{Corollary}[section]
\newtheorem{lemma}{Lemma}[section]
\newtheorem{theorem}{Theorem}[section]

\newtheorem{remark}{Remark}[section]

\newtheorem*{thma}{Theorem A}

\par

\begin{document}

\title{On the growth of Lebesgue constants for convex polyhedra}

\author[Yurii
Kolomoitsev]{Yurii
Kolomoitsev$^{\text{a, b, *, 1}}$}
\address{Institute of Applied Mathematics and Mechanics of NAS of Ukraine,
General Batyuk Str.~19, Slov’yans’k, Donetsk region, Ukraine, 84100}
\email{kolomus1@mail.ru}

\author[Tetiana
Lomako]{Tetiana
Lomako$^{\text{a, b, 1}}$}
\address{Institute of Applied Mathematics and Mechanics of NAS of Ukraine,
General Batyuk Str.~19, Slov’yans’k, Donetsk region, Ukraine, 84100}
\email{tlomako@yandex.ru}

\thanks{$^\text{a}$Universit\"at zu L\"ubeck,
Institut f\"ur Mathematik,
Ratzeburger Allee 160,
23562 L\"ubeck, Germany}

\thanks{$^\text{b}$Institute of Applied Mathematics and Mechanics of NAS of Ukraine,
General Batyuk Str. 19, Slov’yans’k, Donetsk region, Ukraine, 84100}

\thanks{$^1$Supported by   H2020-MSCA-RISE-2014 Project number 645672  (AMMODIT: "\!Approximation Methods for Molecular Modelling and Diagnosis Tools").}

\thanks{$^*$Corresponding author}


\date{\today}
\subjclass[2010]{42B05, 42B15, 42B08} \keywords{Lebesgue constants, Dirichlet kernel, convex polyhedra}

\begin{abstract}
In the paper, new  estimates of the Lebesgue constant
$$
\mathcal{L}(W)=\frac1{(2\pi)^d}\int_{\T^d}\bigg|\sum_{\bm{k}\in W\cap\Z^d} e^{i(\bm{k},\,\bm{x})}\bigg| {\rm d}{\bm x}
 $$
 for convex polyhedra $W\subset\R^d$  are obtained.
The main result states that if $W$ is a convex polyhedron such that $[0,m_1]\times\dots\times [0,m_d]\subset W\subset [0,n_1]\times\dots\times [0,n_d]$, then
$$
c(d)\prod_{j=1}^d \log(m_j+1)\le \mathcal{L}(W)\le C(d)s\prod_{j=1}^d \log(n_j+1),
$$
where $s$ is size of the triangulation of $W$.
\end{abstract}

\maketitle

\section{Introduction}

Estimates of the Lebesgue constants play an important role in the summation of Fourier series, approximation and interpolation theory, and other branches of analysis. Different asymptotic formulas as well as upper and lower estimates of the Lebesgue constants on the $d$-dimensional torus $\T^d$ have been known for years (see~\cite{D2},~\cite{L}, and~\cite[Ch. 9]{TB}).

In the one-dimensional case, the following asymptotic formula is well-known:
\begin{equation*}
  \frac1{2\pi}\int_{\T^1}\left| \sum_{k=0}^n e^{ikx}\right|{\rm d}x=\frac1{2\pi}\int_{\T^1}\left|\frac{\sin((n+1)x/2)}{\sin (x/2)} \right|{\rm d}x \backsimeq\frac{4}{\pi^2}\log n.
\end{equation*}
There are numerous generalizations of this result to the multidimensional case. As a rule one takes some set $W\subset\R^d$ and defines the Lebesgue constant by
$$
\mathcal{L}(W):=\frac1{(2\pi)^d}\int_{\T^d}\bigg|\sum_{\bm{k}\in W\cap\Z^d} e^{i(\bm{k},\,\bm{x})}\bigg| {\rm d}{\bm x}.
$$


The following important result was proved by Belinsky~\cite{Be} (see also~\cite{P} and~\cite{Ba}).

\begin{thma}\label{thA}
 For any convex $d$-dimensional polyhedron $W\subset \R^d$ and $n\ge 1$, there exist two positive constants $C_1$ and $C_2$ such that
  \begin{equation}\label{eqBe}
    C_1(W)\log^d(n+1)\le \mathcal{L}(nW)\le C_2(W)\log^d(n+1).
  \end{equation}
\end{thma}

See also in~\cite{AC} an analog of this theorem for $L_p$ Lebesgue constants.

The following question  seems to be very natural.
Is it possible to write a certain asymptotic relation instead of the ordinal estimate~\eqref{eqBe} or at least to find good estimates for the constants $C_1$ and $C_2$?

It turns out that asymptotic relations for $\mathcal{L}(nW)$ can be obtained only for some special polyhedra with good arithmetical properties. For example, it follows from the result of Skopina~\cite{Sk} (see also~\cite{NP}, \cite{P2}) that if slopes of sides of the $s$-sided convex polyhedron $W\subset\R^2$ are rational and this polyhedron has no parallel sides, then
\begin{equation}\label{eqSkop}
  \mathcal{L}(nW)=\frac{2s}{\pi^3}\log^2n+\mathcal{O}(\log n),
\end{equation}
where $\mathcal{O}$ depends on $W$.

At the same time, if $W$ has irrational slopes of sides, then asymptotics~\eqref{eqSkop} does not hold in general (see \cite{NP}, \cite{P2}).
It is also unclear how $\mathcal{O}$ depends on $W$.

Nevertheless, it is possible to find good estimates of the constant $C_2(W)$ in~\eqref{eqBe}. It is known (see~\cite{YuYu}) that if $W$ is an arbitrary $s$-sided polyhedron in $\R^2$ of diameter $n\ge 1$, then
\begin{equation}\label{eqYu2}
  \mathcal{L}(W)\le Cs\log^2 (n+1),
\end{equation}
where $C$ is some absolute constant.

In special cases, it is possible to improve asymptotics~\eqref{eqSkop} and inequalities~\eqref{eqBe} and~\eqref{eqYu2}.
The simplest case is that $W=R_{n_1,\dots,n_d}=[0,n_1]\times\dots\times [0,n_d]$. By the corresponding one-dimensional result, one has
\begin{equation}\label{Rectan+}
\mathcal{L}(R_{n_1,\dots,n_d})\backsimeq\prod_{j=1}^d\(\frac{4}{\pi^2}\log n_j\).
\end{equation}

For other types of polyhedra $W$, the problem becomes more complicated and has been considered mainly in the case $d=2$.
Let us mention the result of Kuznetsova~\cite{Ku} for the Lebesgue constant of the rhomb
$$
\D_{n_1,n_2}=\left\{(\xi_1,\xi_2)\in \R^2\,:\,\frac{|\xi_1|}{n_1}+\frac{|\xi_2|}{n_2}\le 1\right\}.
$$
It was proved that the asymptotic equality
  \begin{equation}\label{eqKuz}
    \mathcal{L}(\D_{n_1,n_2})=\frac{32}{\pi^4}\log n_1\log n_2-
    \frac{16}{\pi^4}\log^2 n_1+\mathcal{O}(\log n_2)
  \end{equation}
holds uniformly with respect to all natural $n_1$, $n_2$, and $l=n_2/n_1$.

What differentiates  this result
from many others is that no dilations of a certain fixed domain are taken.
Note that nothing is known about analogs  of~\eqref{eqKuz} for $l$ other than integer and the case of several variables, $d\ge 3$.

In this paper,  we obtain the following improvement and generalization of~\eqref{eqBe} and~\eqref{eqYu2} related to the formulas~\eqref{Rectan+} and~\eqref{eqKuz}. We prove that if $W$ is a bounded convex polyhedron in $\R^d$ such that
$$
[0,m_1]\times\dots\times [0,m_d]\subset W\subset [0,n_1]\times\dots\times [0,n_d],
$$
then for sufficiently large $(n_1,\dots,n_d)$ we have
  \begin{equation}\label{T1intr}
    c(d)\prod_{j=1}^d\log (m_j+1)\le\mathcal{L}(W)\le C(d)s\prod_{j=1}^d\log (n_j+1),
  \end{equation}
where $s$ is size of some triangulation of $W$ (see Theorem~\ref{cor1} and Theorem~\ref{thBel} below).

Recall that the size of a triangulation is the number of tetrahedra (simplices) in the triangulation. It is well-known that any convex polyhedron $W$ with $m$ vertices can be represented as a union of at most $\mathcal{O}(m)$ tetrahedra $T_j$, $j=1,\dots,\mathcal{O}(m)$, such that
$T_j\cap T_i$, $i\neq j$, is either empty or a face of both tetrahedra (see~\cite{BEG}, see also~\cite[p. 842]{RMGGS}).

 It is easy to see that \eqref{T1intr} complements~\eqref{eqKuz} in the case of several variables and yields a sharper version of~\eqref{eqBe} and~\eqref{eqYu2} for some  classes of polyhedra. For example, if
$$
\D_{\bm{n}}=\left\{\bm{\xi}\in \R_+^d\,:\,\sum_{j=1}^d\frac{\xi_j}{n_j}\le 1\right\},
$$
then
\begin{equation*}
 c(d)\prod_{j=1}^d\log (n_j+1)\le \mathcal{L}(\D_{\bm{n}})\le C(d)\prod_{j=1}^d\log (n_j+1),
\end{equation*}
where $c$ and $C$ are some positive constants depending only on $d$.

In this paper, we also obtain new estimates of growth of the $L_p$ Lebesgue constants for convex polyhedra (see Theorem~\ref{thMp} and Theorem~\ref{corMp} below). These estimates represent improvements of the corresponding results of the papers~\cite{Ash} and \cite{AC}.

Finally, let us note that the results of this paper can be applied to the multivariate interpolation on the Lissajous-Chebyshev nodes (see~\cite{DE}). In particular, if $d=2$, then the two-sided inequality~\eqref{T1intr}, see also (\ref{M1QT}) below, gives new sharp estimates for the error of approximation of functions by polynomials of the bivariate Lagrange interpolation at the node points of the Lissajous curves (see~\cite{E}).

\subsection {Work organization}
The paper is organized as follows: In Section 2 we provide the basic notation and preliminary remarks needed everywhere below. In Section 3 we
collect auxiliary results.  In Section 4 we prove the main results of the article and provide some examples of their applications to particular classes of polyhedra. Section 5 is devoted to the $L_p$ Lebesgue constants of convex polyhedra.

\section{Basic notations and preliminary remarks}

Let $\T^d \simeq (-\pi,\pi]^d$, $d=1,2,\dots$, be the $d$-dimensional torus.
We use the following notation
$\bm{x}^d=(x_1,\dots,x_d)\in \T^{d},$
$
\bm{k}^d=(k_1,\dots,k_d)\in \Z_+^{d},
$
$
\bm{\xi}^d=(\xi_1,\dots,\xi_d)\in \R_+^d,
$
$
(\bm{x}^d,\bm{k}^d)=k_1x_1+\dots+k_dx_d,
$
and
$$
\Vert f\Vert_{L_p(\T^d)}=\(\int_{\T^d}|f({\bm x}^d)|^p {\rm d}{\bm x}^d\)^\frac1p,\quad 1\le p<\infty.
$$
Denote
$
\bm{n}^d=(n_1,\dots,n_d)\in \R^d,
$
$
\bm{m}^{(d)}=(m_1^{(d)},\dots,m_d^{(d)})\in \R^d,
$
and
$$
\bm{M}^{(d)}= \begin{pmatrix}
n_1 & 0 & 0 & \cdots & 0 \\
n_2 & m_{1}^{(1)} & 0 & \cdots & 0 \\
n_3 & m_{1}^{(2)} & m_{2}^{(2)} & \cdots & 0 \\
\vdots & \vdots & \vdots & & \vdots \\
n_{d} & m_{1}^{(d-1)} & m_{2}^{(d-1)} & \cdots & m_{d-1}^{(d-1)}
\end{pmatrix}.
$$
With such vectors $\bm{n}^d$, $\bm{m}^{(d)}$, and the matrix $\bm{M}^{(d)}$ we associate the following vector function
$$
\bm{\L}^d=(\L_1,\dots,\L_d)\,:\, \R^{d-1}_+\mapsto \R^{d},
$$
where
$
\L_1=n_1
$
and
\begin{equation}\label{FORML}
  \L_s=\L_s(\bm{\xi}^{s-1})=\L_s(\bm{\xi}^{s-1};\bm{M}^{(s)}):=n_s-(\bm{m}^{(s-1)},\bm{\xi}^{s-1}),\quad s=2,3\dots.
\end{equation}


By $P(\bm{\L}^d)$ we denote a polyhedron in $\R^d$ which is defined as a set of vectors $\bm{\xi}^d$ satisfying the system
\begin{equation}\label{sys}
\left\{
  \begin{array}{ll}
    0  \le  \xi_1\le \L_1, \\
    0  \le  \xi_s\le \L_s(\bm{\xi}^{s-1}), & \hbox{$s=2,\dots,d,$}
  \end{array}
\right.
\end{equation}
(see also~\eqref{sys<} below).
In particular, if
the matrix $\bm{M}^{(d)}$ is such that $m_j^{(s)}=0$ for $j=1,\dots,s$ and $s=1,\dots,d-1$,
then $P(\bm{\L}^d)=[0,n_1]\times\dots\times[0,n_d]$ is a rectangle. At the same time, if
$$
\bm{M}^{(d)}= \begin{pmatrix}
n_1 & 0 & 0 & \cdots & 0 \\
n_2 & {n_2}/{n_1} & 0 & \cdots & 0 \\
n_3 & {n_3}/{n_1} & {n_3}/{n_2} & \cdots & 0 \\
\vdots & \vdots & \vdots &  & \vdots \\
n_{d} & {n_d}/{n_1} & {n_3}/{n_2} & \cdots & {n_d}/{n_{d-1}}
\end{pmatrix},
$$
then $$
P(\bm{\L}^d)=\D_{\bm{n}}=\left\{\bm{\xi}\in \R_+^d\,:\,\sum_{j=1}^d\frac{\xi_j}{n_j}\le 1\right\}
$$
is a $d$-dimensional tetrahedron.

The floor, the ceiling, and the fractional part functions are as usual defined by
$$
\lfloor x \rfloor =\max\{m\in \Z\,:\, m\le x\},
$$
$$
\lceil x \rceil =\min\{n\in \Z\,:\, n\ge x\},
$$
and
$$
\{x\}=x-\lfloor x \rfloor,
$$
correspondingly.

By $[x]$ we denote $\lfloor x \rfloor$ or $\lceil x \rceil$. If necessary, we will specify in the corresponding line the meaning of $[\,\cdot\,]$. In a similar manner,  by $\langle x \rangle$  we denote $\{x\}$ or $x-\lceil x \rceil$.

One of the main objects of this paper is the following Dirichlet type kernel:
$$
D_{ \bm{M}^{(d)}}(\bm{x}^d):=\sum_{\bm{k}^d=0}^{[\bm{\L}^d]}e^{i(\bm{k}^d,\,\bm{x}^d)}=\sum_{k_1=0}^{[\L_1]}\sum_{k_2=0}^{[\L_2(k_1)]}\dots \sum_{k_d=0}^{[\L_d(\bm{k}^{d-1})]}e^{i(\bm{k}^d,\,\bm{x}^d)}.
$$

Note that, if $P(\bm{\L}^d)$ is defined by~\eqref{sys}, then
\begin{equation*}
  \begin{split}
    \sum_{\bm{k}^d\in P(\bm{\L}^d) \cap \Z_+^d}e^{i(\bm{k}^d,\,\bm{x}^d)}
&=\sum_{k_1=0}^{[\L_1]}\sum_{k_2=0}^{[\L_2(k_1)]}\dots \sum_{k_d=0}^{[\L_d(\bm{k}^{d-1})]}e^{i(\bm{k}^d,\,\bm{x}^d)}\\
&=\sum_{k_1=0}^{\lfloor \L_1 \rfloor}\sum_{k_2=0}^{\lfloor \L_2(k_1)\rfloor}\dots \sum_{k_d=0}^{\lfloor\L_d(\bm{k}^{d-1})\rfloor}e^{i(\bm{k}^d,\,\bm{x}^d)}.
   \end{split}
\end{equation*}
At the same time, if $P(\bm{\L}^d)$ is defined as a set of vectors $\bm{\xi}^d=(\xi_1,\dots,\xi_d)$ which satisfy the system
\begin{equation}\label{sys<}
\left\{
  \begin{array}{ll}
    0  \le  \xi_1< \L_1+1, \\
    0  \le  \xi_s< \L_s(\bm{\xi}^{s-1})+1, & \hbox{$s=2,\dots,d$},
  \end{array}
\right.
\end{equation}
then
\begin{equation*}
  \begin{split}
    \sum_{\bm{k}^d\in P(\bm{\L}^d) \cap \Z_+^d}e^{i(\bm{k}^d,\,\bm{x}^d)}
=\sum_{k_1=0}^{\lceil \L_1 \rceil}\sum_{k_2=0}^{\lceil \L_2(k_1)\rceil}\dots \sum_{k_d=0}^{\lceil\L_d(\bm{k}^{d-1})\rceil}e^{i(\bm{k}^d,\,\bm{x}^d)}.
   \end{split}
\end{equation*}

Throughout the paper, we suppose that $\sum_{{\bm k}\in \varnothing}=0$ and $\sum_{k=A}^B(\dots)=0$ if $A>B$.
We always take into account this remark, when using the equality
\begin{equation}\label{equ}
  \sum_{k=A}^B(\dots)=\sum_{k=0}^B(\dots)-\sum_{k=0}^{A-1}(\dots)
\end{equation}
for $0\le A<B$.

Let $d=2,3,\dots$. Denote
\begin{equation*}
  \begin{split}
     &G_{\bm{M}^{(d)}}(\bm{x}^d)\\
&:=\frac1{e^{ix_d}-1}\(e^{i(n_d+1)x_d}D_{\bm{M}^{(d-1)}}(\bm{x}^{d-1}-{\bm{m}^{(d-1)}}x_d)-D_{ \bm{M}^{(d-1)}}(\bm{x}^{d-1})\)
   \end{split}
\end{equation*}
and
\begin{equation*}
  \begin{split}
     F_{\bm{M}^{(d)}}(\bm{x}^d)
&:=\frac{e^{i(n_d+1)x_d}}{e^{ix_d}-1}
\sum_{\bm{k}^{d-1}=0}^{[\bm{\L}^{d-1}]}
e^{i(\bm{k}^{d-1},\,\bm{x}^{d-1}-{\bm{m}^{(d-1)}}x_d)}
\(e^{-i \langle \L_d(\bm{k}^{d-1})\rangle x_d}-1\),
   \end{split}
\end{equation*}
where $\bm{m}^{(s)}u:=(m_1^{(s)} u,\dots, m_s^{(s)}u)$ and $u\in \T^1$.

By $\deg_l T$, $l=1,\dots,d$, we denote the order of a trigonometric polynomial $T(\bm{x}^d)$ in the variable $x_l$.
Throughout the paper we use the notation $\, A \lesssim B$ for the
estimate $\, A \le c\, B$, where $A$ and $B$ are some nonnegative functions and $c$ is a positive constant independent of
the appropriate variables in $\, A$ and $\, B.$  Up to Section~\ref{SecLp} this constant $c$ depends only on the dimension $d$. Below $A\asymp B$ means that $A\lesssim B$ and $B\lesssim A$ simultaneously.
In what follows the sign "$\preccurlyeq$" means "$<$" or "$\le$". The concrete meaning of "$\preccurlyeq$" will be explained in the appropriate place.  By $C(\cdot)$ or $C_j(\cdot)$, $j=1,2,\dots$, we denote some positive constants that depend on indicated parameters.

\section{Auxiliary results}

\begin{lemma}\label{leFG}
Let $d\ge 2$. Then
\begin{equation}\label{FG}
  D_{\bm{M}^{(d)}}(\bm{x}^d)=G_{\bm{M}^{(d)}}(\bm{x}^d)+F_{ \bm{M}^{(d)}}(\bm{x}^d).
\end{equation}
\end{lemma}

\begin{proof}
Note that
\begin{equation}\label{FG1}
  D_{\bm{M}^{(d)}}(\bm{x}^d)= \sum_{\bm{k}^{d-1}=0}^{[\bm{\L}^{d-1}]} e^{i(\bm{k}^{d-1},\,\bm{x}^{d-1})}
\sum_{k_d=0}^{[\L_d(\bm{k}^{d-1})]}e^{ik_dx_d}
\end{equation}
and
\begin{equation}\label{FG2}
  \sum_{k_d=0}^{[\L_d(\bm{k}^{d-1})]}e^{ik_dx_d}=\frac{e^{i\([\L_d(\bm{k}^{d-1})]+1\)x_d}-1}{e^{ix_d}-1}.
\end{equation}
Thus, by using \eqref{FG1} and \eqref{FG2} and  taking into account that
\begin{equation*}
  e^{i\([\L_d(\bm{k}^{d-1})]+1\)x_d}-1=e^{i\(\L_d(\bm{k}^{d-1})+1\)x_d}-1
+e^{i\(\L_d(\bm{k}^{d-1})+1\)x_d}\(e^{-i\langle\L_d(\bm{k}^{d-1})\rangle x_d}-1\),
\end{equation*}
we get~\eqref{FG}.
\end{proof}

\begin{lemma}\label{le4}
Let
$$
S_t(x):=\frac{e^{i(t+1)x}-1}{e^{ix}-1},\quad t>0,\quad x\in \T^1.
$$
Then
\begin{equation}\label{Slog}
  \Vert S_{t}\Vert_{L_1(\T^1)}\lesssim \log (t+1).
\end{equation}
\end{lemma}

\begin{proof}
To prove~\eqref{Slog}, we note that for $|x|\le 1/(t+1)$
\begin{equation}\label{Slog1}
  \left|\frac{e^{i(t+1)x}-1}{e^{ix}-1}\right|\lesssim t+1
\end{equation}
and for $1/(t+1)\le |x|\le \pi$
\begin{equation}\label{Slog2}
  \left|\frac{e^{i(t+1)x}-1}{e^{ix}-1}\right| \le \frac{2}{|e^{ix}-1|}\lesssim \frac1{|x|}.
\end{equation}

Now, by using \eqref{Slog1} and \eqref{Slog2}, we get
\begin{equation*}
  \begin{split}
     \Vert S_{t}\Vert_{L_1(\T^1)}&=\int_{|x|\le 1/(t+1)}+\int_{1/(t+1)\le |x|\le \pi}\\
&\lesssim \int_{|x|\le 1/(t+1)}(t+1){\rm d}x+\int_{1/(t+1)\le |x|\le \pi} \frac{{\rm d}x}{|x|}\lesssim \log(t+1).
   \end{split}
\end{equation*}

\end{proof}

%
%
%
%

\begin{lemma}\label{leG}
Let $N\ge N_l=\deg_l D_{\bm{M}^{(d+1)}}\ge 1$, $l=1,\dots,d+1$. Then
\begin{equation}\label{G}
 \Vert G_{\bm{M}^{(d+1)}} \Vert_{L_1(\T^{d+1})}\lesssim \log (N+1)\Vert D_{\bm{M}^{(d)}} \Vert_{L_1(\T^{d})}.
\end{equation}
\end{lemma}

\begin{proof}
Denote
$$
V_l(\bm{x}^d):=\frac{D_{\bm{M}^{(d)}}(\bm{x}^{d}-{\bm{m}_l^{(d)}}x_{d+1})-D_{ \bm{M}^{(d)}}(\bm{x}^{d}-{\bm{m}_{l+1}^{(d)}}x_{d+1})}{e^{ix_{d+1}}-1},
$$
where
$$
{\bm{m}_l^{(d)}}x_{d+1}:=(\underbrace{0,\dots,0}_{l-1},m_{l}^{(d)}  x_{d+1},\dots,m_d^{(d)} x_{d+1})\in \R_+^{d},\quad l=1,\dots,d,
$$
and
$$
{\bm{m}_{d+1}^{(d)}}x_{d+1}:=(0,\dots,0)\in \R_+^{d}.
$$

One has
\begin{equation*}
  \begin{split}
     G_{\bm{M}^{(d+1)}}(\bm{x}^{d+1})&=
     D_{ \bm{M}^{(d)}}(\bm{x}^{d})S_{n_{d+1}}({x}_{d+1})+
\frac{D_{\bm{M}^{(d)}}(\bm{x}^{d}-{\bm{m}^{(d)}}x_{d+1})-
D_{\bm{M}^{(d)}}(\bm{x}^{d})}{e^{ix_{d+1}}-1}\\
&=D_{\bm{M}^{(d)}}(\bm{x}^{d})S_{n_{d+1}}({x}_{d+1})+\sum_{l=1}^{d}V_l(\bm{x}^d).
   \end{split}
\end{equation*}
Therefore, taking into account that $0\le n_{d+1}\le N$ and Lemma~\ref{le4}, we get
\begin{equation}\label{G1}
  \begin{split}
\Vert G_{\bm{M}^{(d+1)}} \Vert_{L_1(\T^{d+1})}&\le
\Vert  S_{n_{d+1}}\Vert_{L_1(\T^{1})}
\Vert D_{\bm{M}^{(d)}} \Vert_{L_1(\T^{d})}
+\sum_{l=1}^{d}
\Vert V_l \Vert_{L_1(\T^{d+1})}\\
&\lesssim \log (N+1)\Vert D_{\bm{M}^{(d)}} \Vert_{L_1(\T^{d})}+\sum_{l=1}^{d}\Vert V_l \Vert_{L_1(\T^{d+1})}.
   \end{split}
\end{equation}

To estimate $\Vert V_l \Vert_{L_1(\T^{d+1})}$, we denote $m_l'=\max\{|m_l^{(d)}|,1\}$. We have
\begin{equation}\label{G2}
  \Vert V_l \Vert_{L_1(\T^{d+1})}=\int_{|x_{d+1}|\ge 1/(N m_l')} + \int_{|x_{d+1}|< 1/(N m_l')}:=I_1+I_2.
\end{equation}

It is easy to see that $|m_l^{(d)}|\le 2N$. Indeed,
$
0\le \L_{d+1}({\bm k}^d)=n_{d+1}-({\bm m}^{(d)},{\bm k}^d)\le N_{d+1}
$
for all admissible ${\bm k}^d\in\Z_+^d$ (that is for those ${\bm k}^d$ which belong to the region of summation in $\sum_{\bm k^d=0}^{[\bm\L^d]}$). Hence, $|m_l^{(d)}|k_l\le N_{d+1}+n_{d+1}\le 2N$ and we obviously have the desired inequality, from which we derive
\begin{equation}\label{forG2++++++++++++}
  \begin{split}
     \log(N m_l')=\log N+\log\(\max\{|m_l^{(d)}|,1\}\)\le 2\log N+1.
   \end{split}
\end{equation}
Using~\eqref{forG2++++++++++++}, we get
\begin{equation}\label{forG2}
\begin{split}
  I_1&\lesssim \int_{|x_{d+1}|\ge 1/(N m_l')}
\frac{|D_{\bm{M}^{(d)}}(\bm{x}^{d}-{\bm{m}_l^{(d)}}x_{d+1})|+|D_{ \bm{M}^{(d)}}(\bm{x}^{d}-{\bm{m}_{l+1}^{(d)}}x_{d+1})|}{|x_{d+1}|}{\rm d}\bm{x}^{d+1}
\\
&\lesssim
\int_{|x_{d+1}|\ge 1/(N m_l')}
\frac{|D_{\bm{M}^{(d)}}(\bm{x}^{d})|}{|x_{d+1}|}
{\rm d}{x}_{d+1}{\rm d}\bm{x}^{d}
\lesssim (\log N+1) \Vert D_{\bm{M}^{(d)}}\Vert_{L_1(\T^{d})}.
\end{split}
\end{equation}
Now, let us estimate $I_2$. Denote $h={\bm{m}_l^{(d)}}x_{d+1}-\bm{m}_{l+1}^{(d)} x_{d+1}$. By the  classical Bernstein inequality (see~\cite[p. 102]{DeLo}), we get
\begin{equation}\label{BB}
\begin{split}
    \int_{\T^1}| D_{\bm{M}^{(d)}}(\bm{x}^{d}-h)-D_{\bm{M}^{(d)}}(\bm{x}^{d})|{\rm d}x_l&\le |h| \int_{\T^1}\bigg|\frac{\partial}{\partial x_{l}}D_{\bm{M}^{(d)}}(\bm{x}^{d})\bigg|{\rm d}x_l\\
    &\le |h|N_l \int_{\T^1}|D_{\bm{M}^{(d)}}(\bm{x}^{d})|{\rm d}x_l.
\end{split}
\end{equation}
Therefore, by~\eqref{BB}, we have
\begin{equation}\label{forG6}
  \begin{split}
     I_2&\lesssim \int_{|x_{d+1}|< 1/(N m_l')}\frac{{\rm d}x_{d+1}}{|x_{d+1}|}\int_{\T^d}
     |D_{\bm{M}^{(d)}}(\bm{x}^{d}-h)-D_{\bm{M}^{(d)}}(\bm{x}^{d})|{\rm d}\bm{x}^d  \\
&\lesssim  \frac{N_l|m_{l}^{(d)}|}{N m_l'}\Vert D_{\bm{M}^{(d)}}\Vert_{L_1(\T^{d})}\lesssim \Vert D_{\bm{M}^{(d)}}\Vert_{L_1(\T^{d})}.
   \end{split}
\end{equation}
Combining~\eqref{G2}, \eqref{forG2}, and~\eqref{forG6}, we obtain
\begin{equation}\label{eqGGGGGGGGGGG}
  \Vert V_l \Vert_{L_1(\T^{d+1})}\lesssim (\log N+1)\Vert D_{\bm{M}^{(d)}}\Vert_{L_1(\T^{d})}
\end{equation}
for each $l=1,\dots,d$.

Finally, combining \eqref{G1} and \eqref{eqGGGGGGGGGGG}, we get~\eqref{G}.

\end{proof}

Everywhere below, we denote
$
\bm{n}_j^s=(n_{j1},\dots,n_{js})\in \R^s,
$
$
\bm{m}_j^{(s)}=(m_{j1}^{(s)},\dots,m_{js}^{(s)})\in \R^s,
$
$$
\bm{M}_j^{(s)}= \begin{pmatrix}
n_{j1} & 0 & 0 & \cdots & 0 \\
n_{j2} & m_{j1}^{(1)} & 0 & \cdots & 0 \\
n_{j3} & m_{j1}^{(2)} & m_{j2}^{(2)} & \cdots & 0 \\
\vdots & \vdots & \vdots & & \vdots \\
n_{js} & m_{j1}^{(s-1)} & m_{j2}^{(s-1)} & \cdots & m_{js-1}^{(s-1)}
\end{pmatrix},
$$
$$
\bm{\L}_j^s=(\L_{j1},\dots,\L_{js})\,:\, \R^{s-1}_+\mapsto \R^{s},
$$
where
$
\L_{j1}=n_{j1}
$
and
$
  \L_{js}=\L_{js}(\bm{\xi}^{s-1}):=n_{js}-(\bm{m}_j^{(s-1)},\bm{\xi}^{s-1})$ for $s=2,3\dots$.

\begin{lemma}\label{lexs}
Let $\b\in\R$, $\bm{\a}^d=(\a_1,\dots,\a_d)\in \R^d$, and
$$
T_{\bm{M}^{(d)}}(\bm{x}^d)=\sum_{\bm{k}^d=0}^{[\bm{\L}^d]}a_{\bm{k}^d} e^{i(\bm{k}^d,\,\bm{x}^d)},\quad a_{\bm{k}^d}\in \C.
$$
Then
\begin{equation}\label{xs1}
  \sum_{\underset{\b \preccurlyeq ({\bm{\a}^{d}},\,\bm{k}^{d})}{{\bm{k}^d=0}}}^{[\bm{\L}^d]}a_{\bm{k}^d} e^{i(\bm{k}^d,\,\bm{x}^d)}
=\sum_{j=1}^R \e_j T_{\bm{M}_j^{(d)}}(\bm{x}^d),
\end{equation}
where
$R\le C(d)$, $\e_j\in\{-1,1\}$, and
$$
T_{\bm{M}_j^{(d)}}(\bm{x}^d)=\sum_{\bm{k}^d=0}^{[\bm{\L}_j^d]}a_{\bm{k}^d}e^{i(\bm{k}^d,\,\bm{x}^d)}
$$
are such that
$\deg_l T_{\bm{M}_j^{(d)}}\le \deg_l T_{\bm{M}^{(d)}}$ for all $j=1,\dots,R$ and $l=1,\dots,d$.
\end{lemma}

\begin{proof}
The system
\begin{equation*}
\left\{
  \begin{array}{ll}
    0  \le  k_1\le [\L_1], \\
    0  \le  k_s\le [\L_s(\bm{k}^{s-1})], & \hbox{$s=2,\dots,d,$}\\
    \b \preccurlyeq ({\bm{\a}^{d}},\,\bm{k}^{d})
  \end{array}
\right.
\end{equation*}
implies that
\begin{equation*}
\left\{
  \begin{array}{ll}
    0  \le  k_1\le \L_1\quad (\text{or}\quad 0  \le  k_1< \L_1+1), \\
    0  \le  k_s\le \L_s(\bm{k}^{s-1}) \quad (\text{or}\quad 0  \le  k_s< \L_s(\bm{k}^{s-1})+1), & \hbox{$s=2,\dots,d,$}\\
    \b \preccurlyeq ({\bm{\a}^{d}},\,\bm{k}^{d}).
  \end{array}
\right.
\end{equation*}
To see this, one may use the fact that for $n\in \Z$ the inequality $n\le \lfloor x \rfloor$ is equivalent to $n\le x$ and the inequality $n\le \lceil x \rceil$ is equivalent to $n<x+1$.

Next, by the Fourier–Motzkin elimination method (see~\cite[Ch. 12]{Sc} and~\cite{Sch}), the above system can be rewritten as a union of $r$ ($r\le C(d)$) systems of the following form
\begin{equation*}
  \left\{
  \begin{array}{ll}
    \tilde{\L}_{j1} \preccurlyeq  k_1 \preccurlyeq \tilde{\tilde{\L}}_{j1}, \\
    \tilde{\L}_{js}(\bm{k}^{s-1}) \preccurlyeq  k_s \preccurlyeq \tilde{\tilde{\L}}_{js}(\bm{k}^{s-1}), & \hbox{$s=2,\dots,d,$}
  \end{array}
\right.
\end{equation*}
where $\tilde{\L}_{js}$ and $\tilde{\tilde{\L}}_{js}$, $j=1,\dots,r$, have the form~\eqref{FORML}.
Therefore, one has
\begin{equation}\label{sumr}
    \sum_{\underset{\b \preccurlyeq ({\bm{\a}^{d}},\,\bm{k}^{d})}{{\bm{k}^d=0}}}^{[\bm{\L}^d]}a_{\bm{k}^d} e^{i(\bm{k}^d,\,\bm{x}^d)}
=\sum_{j=1}^r \sum_{\tilde{\L}_{j1} \preccurlyeq  k_1 \preccurlyeq \tilde{\tilde{\L}}_{j1}}\dots\sum_{\tilde{\L}_{js}(\bm{k}^{s-1}) \preccurlyeq  k_s \preccurlyeq \tilde{\tilde{\L}}_{js}(\bm{k}^{s-1})}a_{\bm{k}^d} e^{i(\bm{k}^d,\,\bm{x}^d)}.
\end{equation}
Now,~\eqref{sumr} and  equality~\eqref{equ}
imply~\eqref{xs1}.
\end{proof}

\begin{lemma}\label{lemod}
One has
\begin{equation}\label{mod0}
\Bigg\Vert\sum_{\underset{\L_{d+1}(\bm{k}^d)\in\Z_+}{\bm{k}^d=0}}^{[\bm{\L}^d]}
e^{i(\bm{k}^d,\,\bm{x}^d)}\Bigg\Vert_{L_1(\T^{d})}\le
\sum_{j=1}^R \Vert
D_{\bm{{M}}_j^{(d)}}\Vert_{L_1(\T^{d})},
\end{equation}
where
$R\le C_1(d)$
and the matrices $\bm{{M}}_j^{(d)}$, $j=1,\dots,R$, are such that
\begin{equation}\label{eqDEG}
  \deg_l D_{\bm{M}_j^{(d)}}\le C_2(d)\deg_l D_{\bm{M}^{(d)}}\quad \text{for all}\quad l=1,\dots,d\quad \text{and}\quad j=1,\dots,R.
\end{equation}
\end{lemma}

\begin{proof}
Let us prove the following equality, from which \eqref{mod0} can be easily derived,
\begin{equation}\label{mod}
S:=\sum_{\underset{\L_{d+1}(\bm{k}^d)\in\Z_+}{\bm{k}^d=0}}^{[\bm{\L}^d]}
e^{i(\bm{k}^d,\,\bm{x}^d)}=
\sum_{j=1}^R \e_j
e^{i(\bm{\b}_j^d,\,\bm{x}^d)}
D_{\bm{{M}}_j^{(d)}}
(\bm{r}_j^d\bm{x}^d+\bm{\a}_{j}^{(1)} x_2+\dots+\bm{\a}_{j}^{(d-1)} x_d),
\end{equation}
where
$\e_j\in \{-1,0,1\}$, $\bm{\b}_{j}^{d}\in \Z^d$, $\bm{r}_j^d\in \N^d$, $\bm{r}_j^d\bm{x}^d=(r_{j1}x_1,\dots,r_{jd}x_d)$,
and $\bm{\a}_j^{(l)}=(\a_{j1}^{(l)},\dots,\a_{jl}^{(l)},0,\dots,0)\in \R^d$ for all $j=1,\dots,R$ and $l=1,\dots,d-1$.


We  prove \eqref{mod} by using induction.
First, let $d=1$ and
$$
S=\sum_{\underset{n_2-m_1 k\in \Z_+}{k=0}}^{n_1} e^{ikx}.
$$
Consider four cases for the parameters $n_2$ and $m_1$:

1) If $n_2\not\in \Q$ and $m_1\in \Q$, then it is obvious that $S=0$.

2) If $n_2 \not\in \Q$ and $m_1 \not\in \Q$, then
there are only two possibilities: $S=0$ or there exists at most one integer
$\b\in [0,n_1]$ such that $n_2-\b m_1\in \Z_+$. Otherwise, if there existed also an integer $\b'\in [0,n_1]$ such that $n_2-\b'm_1\in \Z_+$, then we would have that $(\b'-\b)m_1\in \Z_+$, which is impossible. Therefore, we have
$$
S=\left\{
    \begin{array}{ll}
      e^{i \b x}, & \hbox{if there exists $\b\in [0,n_1]\cap \Z$:  $n_2-\b m_1\in \Z_+$,} \\
      0, & \hbox{otherwise.}
    \end{array}
  \right.
$$

3) If $n_2 \in \Q$ and $m_1 \not\in \Q$, then it is easy to see that
$$
S=\left\{
    \begin{array}{ll}
      1, & \hbox{$n_2\in \Z_+$,} \\
      0, & \hbox{otherwise.}
    \end{array}
  \right.
$$

4) Finally, let $n_2, m_1\in \Q$ be such that $m_1={p_1}/{q}$ and $n_2={p_2}/{q}$, where $q\in \N$. Then the condition $n_2-m_1 k\in \Z$ is equivalent to
\begin{equation}\label{pq}
  p_1 k\equiv p_2\pmod q.
\end{equation}

Let $c=\gcd(p_1,q)$ (the greatest common divisor). If $c\neq 1$ and $c \nmid p_2$ ($c$ does not divide $p_2$), then \eqref{pq} does not have any solution and one can put $S=0$. If $c=1$ or $c \mid p_2$ ($c$ divides $p_2$), then the solution of \eqref{pq} can be represented as
$$
k=\b+r\nu,\quad 0\le \b<r,\quad \quad \nu\in\Z.
$$
From the inequality $0\le k\le n_1$, we get that $-\lfloor A\rfloor\le \nu\le \lfloor B\rfloor$, where $A=\b/ r$ and $B={(n_1-\b)}/{r}$. Therefore,
\begin{equation*}
  \begin{split}
     S=\sum_{\nu=0}^{\lfloor B\rfloor}e^{i(\b+r\nu)x}=e^{i\b x}\sum_{\nu=0}^{\lfloor B\rfloor}e^{ir\nu x},
   \end{split}
\end{equation*}
which implies~\eqref{mod} in the case $d=1$.

Now, let us fix $d$ and assume that~\eqref{mod}  holds in any dimension less than $d$.
As in the above case $d=1$,  we consider several cases for the parameters in the following condition:
\begin{equation}\label{co}
  \L_{d+1}(\bm{k}^d)=n_{d+1}-\sum_{l=1}^d m_lk_l\in \Z_+.
\end{equation}

1) If $n_{d+1}\not\in\Q$ and $m_l\in\Q$, $l=1,\dots,d$, then condition~\eqref{co} implies that $S=0$.

2) Let $n_{d+1}\not\in\Q$ and $m_l\not\in\Q$, $l=1,\dots,d$. In this case we have two possibilities: $S=0$ or there exists a non-zero vector $\bm k_1=(k_{1,1},\dots,k_{1,d})\in\Z_+^d$ such that
\begin{equation}\label{++1}
  n_{d+1}-\sum_{l=1}^d m_l k_{1,l}=N_1\in \Z_+.
\end{equation}
In the last case, supposing that $k_{1,d}\neq 0$, we obtain from~\eqref{++1} that
\begin{equation}\label{++2}
  m_d=a_1 n_{d+1}+\sum_{l=1}^{d-1}b_{1,l}m_l+c_1,
\end{equation}
where $a_1=1/k_{1,d}$ and $c_1, b_{1,l}\in \Q$, $l=1,\dots,d-1$. Then we derive from~\eqref{++2} that in the considered case, \eqref{co} is equivalent to
\begin{equation*}
  (1-a_1k_d)n_{d+1}-\sum_{l=1}^{d-1}m_l (k_l+b_{1,l}k_d)-c_1k_d\in \Z_+.
\end{equation*}
We again have two possibilities: $S=0$ or there exists a non-zero vector $\bm k_2=(k_{2,1},\dots,k_{2,d})\in \Z_+^d$ such that
\begin{equation}\label{++3}
  (1-a_1k_d)n_{d+1}-\sum_{l=1}^{d-1}m_l (k_{2,l}+b_{1,l}k_{2,d})-c_1k_{2,d}=N_2\in \Z_+.
\end{equation}
Supposing that $k_{2, d-1}+b_{1, d-1}k_{2, d-1}\neq 0$, we derive from~\eqref{++3} that
\begin{equation*}
  m_{d-1}=a_2 n_{d+1}+\sum_{l=1}^{d-2}b_{2,l}m_l+c_2,
\end{equation*}
where $0\neq a_2\in \Q$ and $c_2,b_{2,l}\in \Q$, $l=1,\dots,d-2$.

Repeating this procedure ($d$ times if necessary), in the final step we again derive that there are two possibilities: $S=0$ or there exist $\nu_j, \mu_j\in \Q$ such that
\begin{equation}\label{equalities_munu}
  m_j=\nu_j+n_{d+1}\mu_j,\quad j=1,\dots,d.
\end{equation}
It is clear that this representation is unique.

Next, from~\eqref{co} and~\eqref{equalities_munu}, we get
$$
n_{d+1}\(1-\sum_{l=1}^d\mu_l k_l\)-\sum_{l=1}^d \nu_l k_l\in\Z_+.
$$
In view of $n_{d+1}\not\in \Q$, we derive that this condition is possible only if
\begin{equation}\label{000}
  \sum_{l=1}^d \mu_l k_l =1\quad\text{and}\quad \sum_{l=1}^d \nu_l k_l\in \Z_+.
\end{equation}
It is clear that $\mu_d\neq 0$. Thus, the first formula in~\eqref{000} yields that
$k_d=\b-(\bm{\a}^{d-1},\bm{k}^{d-1})\in \Z_+$, where $\b={1}/{\mu_d}$ and $\a_l={\mu_l}/{\mu_d}$, $l=1,\dots,d-1$. Combining this with the second formula from~\eqref{000}, we get
\begin{equation}\label{SSS}
  \begin{split}
     S=\sum_{\underset{\b-(\bm{\a}^{d-1},\bm{k}^{d-1})\le [\La_{d}(\bm{k}^{d-1})],\,\widetilde{\b}-(\bm{\widetilde{\a}}^{d-1},\bm{k}^{d-1})\in \Z_+}{\bm{k}^{d-1}=0}}^{[\bm{\La}^{d-1}]} e^{i(\bm{k}^{d-1},\,\bm{x}^{d-1})}e^{i(\b-(\bm{\a}^{d-1},\,\bm{k}^{d-1}))x_d}.
   \end{split}
\end{equation}
By Lemma~\ref{lexs}, \eqref{SSS} can be rewritten as
\begin{equation}\label{SSS1}
   S=e^{i\b x_d}\sum_{j=1}^{R'}\e_j'\sum_{\underset{\b-(\bm{\a}^{d-1},\bm{k}^{d-1})\in \Z_+}{\bm{k}^{d-1}=0}}^{[\widetilde{\bm{\La}}_j^{d-1}]} e^{i(\bm{k}^{d-1},\,\bm{x}^{d-1}-\bm{\a}^{d-1}x_d)},
\end{equation}
where $\e_j'\in \{-1,1\}$ and $R'\le C(d)$.
Thus, applying the induction hypothesis to each sum in~\eqref{SSS1}, we obtain~\eqref{mod}.

3) Let us consider the case $n_{d+1}\in\Q$, $m_{l_1},\dots m_{l_t}\not\in\Q$ for some $t\le d$ and $1\le l_1<\dots<l_t\le d$, and $m_k\in \Q$ for $k\neq l_j$, $j=1,\dots,t$. Suppose for simplicity that $m_t,\dots,m_d\not\in\Q$. In this case, the condition $\L_{d+1}(\bm{k}^{d})\in\Z_+$ implies
\begin{equation}\label{3+1}
  \sum_{l=t}^d m_l k_l\in \Q.
\end{equation}
It is clear that~\eqref{3+1} holds for $k_t=\dots=k_d=0$. If there are no other admissible $k_t,\dots, k_d$ such that~\eqref{3+1} is fulfilled, then~\eqref{co} is equivalent to the following condition:
$$
n_{d+1}-\sum_{l=1}^{t-1}m_l k_l\in \Z_+\quad\text{and}\quad k_t=\dots=k_d=0.
$$
If $t=1$, then  this condition implies that $S=1$ if $n_{d+1}\in \Z_+$ and $S=0$ otherwise. In the case $t>1$, the above condition implies that
\begin{equation}\label{3+2}
  S=\sum_{\underset{n_{d+1}-(\bm m^{t-1}, \bm k^{t-1})\in\Z_+}{\bm{k}^{t-1}=0}}^{[\bm{\L}^{t-1}]}
e^{i(\bm{k}^{t-1},\,\bm{x}^{t-1})}.
\end{equation}
Thus, applying the induction hypothesis to~\eqref{3+2}, we derive~\eqref{mod}. Note that we have the same conclusion in the case $t=d$.

Let us suppose that $t<d$ and there exists a non-zero vector $(k_{1,t},\dots,k_{1,d})\in \Z_+^{d-t+1}$ such that~\eqref{3+1} is fulfilled. Let, for example, $k_{1,d}\neq 0$. Then, from~\eqref{3+1} it follows that
\begin{equation*}
m_d=a_{1,d-1}m_{d-1}+a_{1,d-2}m_{d-2}+\dots+a_{1,t}m_t+c_1,
\end{equation*}
where $c_1, a_{1,l}\in \Q$, $l=t,\dots,d-1$. Thus,~\eqref{3+1} can be rewritten in the following form:
\begin{equation}\label{3+4}
\sum_{l=t}^{d-1} m_l (k_l+a_{1,l}k_d)\in \Q.
\end{equation}

As above, let us consider two cases for~\eqref{3+4}. First, let~\eqref{3+4} holds only if
\begin{equation*}
  k_l+a_{1,l} k_d=0,\quad l=t,\dots,d-1.
\end{equation*}
It is clear that for some $l_0\in \{t,\dots,d-1\}$ one has $a_{1,l_0}\neq 0$. Let, for simplicity, $l_0=d-1$. Then we derive
that $k_d=a_{1,d-1}^{-1}=\g k_{d-1}$. Thus, in this case,~\eqref{co} is equivalent to
$$
n_{d+1}-\sum_{l=1}^{d-1}m_l k_l-m_d\g k_{d-1}\in \Z_+.
$$
Hence,
\begin{equation}\label{3+6}
S=\sum_{\underset{\g k_{d-1}\le [\L_d(\bm k^{d-1})],\,  n_{d+1}-(\bm m^{d-1}, \bm k^{d-1})-m_d\g k_{d-1}\in\Z_+}{\bm{k}^{d-1}=0}}^{[\bm{\L}^{d-1}]}
e^{i(\bm{k}^{d-1},\,\bm{x}^{d-1})}e^{i\g k_{d-1}x_d}.
\end{equation}
Thus, applying Lemma~\ref{lexs} and the induction hypothesis to~\eqref{3+6}, we derive~\eqref{mod}.

Now, let us consider the case of an existing non-zero vector $(k_{2,t},\dots,k_{2,d})\in \Z_+^{d-t+1}$ such that
\begin{equation}\label{3+7}
  \sum_{l=t}^{d-1} m_l(k_{2,l}+a_{1,l}k_{2,d})\in \Q
\end{equation}
and for some $l_0\in \{t,\dots,d-1\}$
\begin{equation}\label{3+8}
  k_{2,l_0}+a_{1,l_0}k_{2,d}\neq 0.
\end{equation}
Let, for example,~\eqref{3+8} holds for $l_0=d-1$. In this case, combining~\eqref{3+7} and~\eqref{3+8}, we derive
\begin{equation*}
  m_{d-1}=a_{2,d-1}m_{d-2}+\dots+a_{2,t}m_t+c_2,
\end{equation*}
where $c_2, a_{2,l}\in \Q$, $l=t,\dots,d-2$. Thus, \eqref{3+4} can be rewritten in the following equivalent form:
\begin{equation*}
  \sum_{l=t}^{d-2} m_l\(k_l+a_{2,l}k_{d-1}+(a_{2,l}a_{1,d-1}+a_{1,l})k_d\)\in \Q.
\end{equation*}

It remains to apply the previous arguments a necessary number of times.

Other cases for  $m_{l_1},\dots m_{l_t}\not\in \Q$ can be considered in a similar way.

4) The case $n_{d+1}\not\in\Q$, $m_{l_1},\dots m_{l_t}\not\in\Q$ for some $t\le d$ and $1\le l_1<\dots<l_t\le d$, and $m_k\in \Q$ for $k\neq l_j$, $j=1,\dots,t$, can be considered by analogy with the cases 2) and 3).

5) It remains to consider the case $n_{d+1}\in\Q$ and $m_l\in \Q$, $l=1,\dots,d$. We can suppose that $n_{d+1}={p_{d+1}}/{q}$ and $m_l={p_l}/{q}$, $l=1,\dots,d$, where $q\in \N$.

Denote $c_d=\gcd(p_d,q)$. If $c_d=1$ or $c_d\, |\, p_{d+1}-(\bm{p}^{d-1},\bm{k}^{d-1})$,
then, by the well-known formula for the Diophantine equations, the condition $\L_{d+1}(\bm{k}^d)\in \Z$ implies that
\begin{equation}\label{modXXX}
  k_d=b_0-(\bm{a}^{d-1},\,\bm{k}^{d-1})+r\nu,\quad \nu\in\Z,
\end{equation}
where
$$
r=\frac q{c_d}\ge 1,\quad b_0=\frac{p_{d+1}}{c_d}\(\frac{p_d}{c_d}\)^{\vp(\frac q{c_d})-1},\quad a_l=\frac{p_l}{c_d}\(\frac{p_d}{c_d}\)^{\vp(\frac q{c_d})-1},\quad l=1,\dots,d-1,
$$
and $\vp$ is Euler's function.
It is clear that one can rewrite \eqref{modXXX} such that
$$
k_d=b-(\bm{a}^{d-1},\,\bm{k}^{d-1})+r\nu,\quad \nu\in\Z,
$$
and
$$
(\bm{a}^{d-1},\,\bm{k}^{d-1})-b \ge 0
$$
for all admissible $\bm{k}^d\in\Z_+^d$.

Since $0\le k_d\le \L_d(\bm{k}^{d-1})$, we get $$\frac1r\((\bm{a}^{d-1},\bm{k}^{d-1})-b\)\le \nu \le \frac1r\([\L_d(\bm{k}^{d-1})]+(\bm{a}^{d-1},\bm{k}^{d-1})-b\)$$ and, therefore, it follows that
\begin{equation}\label{SSS3}
  \begin{split}
     S&=\sum_{\underset{ (p_{d+1}-(\bm{p}^{d-1},\bm{k}^{d-1}))/c_d\in \Z_+  } {\bm{k}^{d-1}=0}} ^{[\bm{\L}^{d-1}]}  e^{i(\bm{k}^{d-1},\,\bm{x}^{d-1})}\sum_{\underset{k_d=(b-(\bm{a}^{d-1},\,\bm{k}^{d-1}))+r\nu}{k_d=0}}^{[\L_d(\bm{k}^{d-1})]} e^{ik_d x_d}\\
&=e^{i b x_d}\sum_{\underset{  (p_{d+1}-(\bm{p}^{d-1},\bm{k}^{d-1}))/c_d\in \Z_+ } {\bm{k}^{d-1}=0}} ^{[\bm{\L}^{d-1}]}  e^{i(\bm{k}^{d-1},\,\bm{x}^{d-1}-\bm{a}^{d-1}x_d)}\sum_{\nu=\lceil A(\bm{k}^{d-1})\rceil}^{[B(\bm{k}^{d-1})]} e^{ir\nu x_d},
   \end{split}
\end{equation}
where
$$
A(\bm{k}^{d-1})=\frac 1r \( (\bm{a}^{d-1},\,\bm{k}^{d-1})-b\)
$$
and
$$
[B(\bm{k}^{d-1})]=\left\{
    \begin{array}{ll}
      \left\lfloor \frac1r\(\L_d(\bm{k}^{d-1})+(\bm{a}^{d-1},\bm{k}^{d-1})-b\)\right\rfloor, & \hbox{$[\L_d(\bm{k}^{d-1})]=\lfloor \L_d(\bm{k}^{d-1})\rfloor$,} \\
\\
      \left\lceil \frac1r\(\L_d(\bm{k}^{d-1})+(\bm{a}^{d-1},\bm{k}^{d-1})-b+1\)\right\rceil-1, & \hbox{otherwise.}
    \end{array}
  \right.
$$

Next, to ensure~\eqref{eqDEG} we choose $\tilde{b}\in \Z$ and $\bm{\tilde{a}}^{d-1}\in\Z^{d-1}$  such that
$$
dN_d\le \left|\frac{b}r-\tilde{b}\right|\le 1+dN_d
$$
and
$$
dN_d\le\left|\frac{a_l}r-{\tilde{a}}_l\right|\le 1+dN_d,\quad l=1,\dots,d-1,
$$
where $N_d=\deg_d D_{\bm{M}^{(d)}}$. Thus, we get
\begin{equation}\label{ZZZ}
\begin{split}
   \sum_{\nu=\lceil A(\bm{k}^{d-1})\rceil}^{[B(\bm{k}^{d-1})]}
   e^{ir\nu x_d}&=
   e^{ir(\tilde{b}-(\bm{\tilde{a}}^{d-1},\,\bm{k}^{d-1})) x_d}
   \sum_{\nu=\lceil A(\bm{k}^{d-1})\rceil}^{[B(\bm{k}^{d-1})]}
   e^{ir(\nu-(\tilde{b}-(\bm{\tilde{a}}^{d-1},\,\bm{k}^{d-1}))) x_d}\\
   &=e^{ir(\tilde{b}-(\bm{\tilde{a}}^{d-1},\,\bm{k}^{d-1})) x_d}\sum_{\nu=\lceil \tilde{A}(\bm{k}^{d-1})\rceil}^{[\tilde{B}(\bm{k}^{d-1})]} e^{ir\nu x_d},
\end{split}
\end{equation}
where
\begin{equation*}\label{ZZZ1}
  0\le \tilde{A}(\bm{k}^{d-1})\le \tilde{B}(\bm{k}^{d-1})\le 2dN_d.
\end{equation*}
Finally, combining~\eqref{SSS3} and \eqref{ZZZ}, and using~\eqref{equ} and the inductive hypothesis, we obtain~\eqref{mod}.


\end{proof}

\begin{lemma}\label{leD}
Let $r\in \N$ and $\L_{d+1}(\bm{k}^{d})\ge 0$. In terms of Lemma~\ref{lemod}, one has
\begin{equation}\label{D1}
\begin{split}
  &\left\Vert \sum_{\bm{k}^d=0}^{[\bm{\L}^d]}e^{i(\bm{k}^d,\bm{x}^d)}\langle \L_{d+1}(\bm{k}^{d}) \rangle^r \right\Vert_{L_1(\T^d)}\\
&\quad\quad\quad\lesssim \sum_{j=1}^R\Vert D_{\bm{M}_j^{(d)}}\Vert_{L_1(\T^d)}+\sum_{\nu=1}^r\binom{r}{\nu}\left\Vert \sum_{\bm{k}^d=0}^{[\bm{\L}^d]}e^{i(\bm{k}^d,\bm{x}^d)}\{ \L_{d+1}(\bm{k}^{d}) \}^\nu \right\Vert_{L_1(\T^d)}.
\end{split}
\end{equation}
\end{lemma}

\begin{proof} Inequality~\eqref{D1} is obvious for $\langle x\rangle=\{x\}$. Let us consider the case $\langle x\rangle=x-\lceil x\rceil$.
Using  the equality
$$
\langle x\rangle=x-\lceil x\rceil=\left\{
    \begin{array}{ll}
      0, & \hbox{$x\in\Z$,} \\
      \{x\}-1, & \hbox{$x\not\in\Z$,}
    \end{array}
  \right.
$$
we derive
\begin{equation*}
  \begin{split}
     &\sum_{\bm{k}^d=0}^{[\bm{\L}^d]}e^{i(\bm{k}^d,\bm{x}^d)}\langle \L_{d+1}(\bm{k}^{d}) \rangle^r\\
&=\sum_{\bm{k}^d=0}^{[\bm{\L}^d]}e^{i(\bm{k}^d,\bm{x}^d)}\(\{\L_{d+1}(\bm{k}^{d})\}-1\)^r-
(-1)^r\sum_{\underset{\L_{d+1}(\bm{k}^d)\in\Z}{\bm{k}^d=0}}^{[\bm{\L}^d]}e^{i(\bm{k}^d,\bm{x}^d)}\\
&=(-1)^{r+1}\sum_{\underset{\L_{d+1}(\bm{k}^d)\in\Z}{\bm{k}^d=0}}^{[\bm{\L}^d]}e^{i(\bm{k}^d,\bm{x}^d)}
+\sum_{\bm{k}^d=0}^{[\bm{\L}^d]}e^{i(\bm{k}^d,\bm{x}^d)}\\
&\quad\quad\quad\quad\quad\quad\quad\quad\quad+\sum_{\nu=1}^r (-1)^{r-\nu} \binom{r}{\nu}\sum_{\bm{k}^d=0}^{[\bm{\L}^d]}e^{i(\bm{k}^d,\bm{x}^d)}\{ \L_{d+1}(\bm{k}^{d}) \}^\nu.
   \end{split}
\end{equation*}
Therefore,
\begin{equation}\label{DDDDDD1}
\begin{split}
  &\left\Vert \sum_{\bm{k}^d=0}^{[\bm{\L}^d]}e^{i(\bm{k}^d,\bm{x}^d)}\langle \L_{d+1}(\bm{k}^{d}) \rangle^r \right\Vert_{L_1(\T^d)}\le
\Bigg\Vert  \sum_{\underset{\L_{d+1}(\bm{k}^d)\in\Z}{\bm{k}^d=0}}^{[\bm{\L}^d]}e^{i(\bm{k}^d,\bm{x}^d)}\Bigg\Vert_{L_1(\T^d)}\\
&\quad\quad\quad+\Bigg\Vert  \sum_{\bm{k}^d=0}^{[\bm{\L}^d]}e^{i(\bm{k}^d,\bm{x}^d)}\Bigg\Vert_{L_1(\T^d)}+\sum_{\nu=1}^r\binom{r}{\nu}\left\Vert \sum_{\bm{k}^d=0}^{[\bm{\L}^d]}e^{i(\bm{k}^d,\bm{x}^d)}\{ \L_{d+1}(\bm{k}^{d}) \}^\nu \right\Vert_{L_1(\T^d)}.
\end{split}
\end{equation}
It remains to apply Lemma~\ref{lemod} to the first sum  in the right-hand side of~\eqref{DDDDDD1}.
\end{proof}

\begin{lemma}\label{leforF}
One has
\begin{equation*}
 \Vert F_{\bm{M}^{(d+1)}}\Vert_{L_1(\T^{d+1})}\lesssim \sum_{s=1}^\infty \frac{(2\pi)^s}{s!}\left\Vert \sum_{\bm{k}^d=0}^{[\bm{\L}^d]}e^{i(\bm{k}^d,\,\bm{x}^d)}\langle \L_{d+1}(\bm{k}^{d}) \rangle^s \right\Vert_{L_1(\T^d)}.
\end{equation*}
\end{lemma}

\begin{proof}
  We have
\begin{equation*}
\begin{split}
    &\Vert F_{\bm{M}^{(d+1)}}\Vert_{L_1(\T^{d+1})}
    \\
&\lesssim \int_{\T^{d+1}}\frac1{|x_{d+1}|}
\left|
\sum_{\bm{k}^d=0}^{[\bm{\L}^d]}
e^{i(\bm{k}^{d},\, \bm{x}^{d})}
\(e^{-i\langle\L_{d+1}(\bm{k}^{d})\rangle x_{d+1}}-1\)
 \right| {\rm d}\bm{x}^{d+1}
\\
&\lesssim \int_{\T^{d+1}}\frac1{|x_{d+1}|}
\left| \sum_{\bm{k}^d=0}^{[\bm{\L}^d]}
e^{i(\bm{k}^{d},\, \bm{x}^{d})}
\sum_{s=1}^\infty
\frac{(-ix_{d+1})^s}{s!}\langle\L_{d+1}(\bm{k}^{d})\rangle^s  \right| {\rm d}\bm{x}^{d+1}
\\
&\lesssim \sum_{s=1}^\infty \frac{(2\pi)^s}{s!}\int_{\T^{d}}
\left| \sum_{\bm{k}^d=0}^{[\bm{\L}^d]}
e^{i(\bm{k}^{d},\, \bm{x}^{d})}
\langle\L_{d+1}(\bm{k}^{d})\rangle^s
\right| {\rm d}\bm{x}^{d}.
\end{split}
\end{equation*}
 The lemma is proved.
\end{proof}

\begin{lemma}\label{leqq}
Let $s\in\N$ and $\L_{d+1}({\bm k}^d)=n_{d+1}-({\bm m}^{(d)},{\bm k}^d)\ge 0$, where $n_{d+1}={p_{d+1}}/{q}\in \mathbb{Q}$ and $m_l^{(d)}={p_{l}}/{q}\in \mathbb{Q}$, $l=1,\dots,d$, with $q\in \N$. Then
\begin{equation}\label{qq}
 \left\Vert \sum_{\bm{k}^d=0}^{[\bm{\L}^d]}e^{i(\bm{k}^d,\,\bm{x}^d)}\{ \L_{d+1}(\bm{k}^{d}) \}^s \right\Vert_{L_1(\T^d)}
\lesssim \log(s(q+1))\Vert D_{\bm{M}^{(d)}}\Vert_{L_1(\T^{d})}.
\end{equation}
\end{lemma}

\begin{proof}
To prove \eqref{qq}, let us consider the following auxiliary 1-periodic function:
$$
h(u):=\left\{
        \begin{array}{ll}
          u^s, & \hbox{$0\le u\le 1-\frac1{q}$,} \\
          \(1-\frac1{q}\)^s q(1-u), & \hbox{$1-\frac1{q}\le u\le 1$.}
        \end{array}
      \right.
$$
One has (see~\cite[p. 1063]{Ash})
$$
|\widehat{h}(k)|\lesssim \frac1{|k|}\quad\text{and}\quad |\widehat{h}(k)|\lesssim \frac{s q}{|k|^2},\quad k\in \Z,
$$
and
\begin{equation}\label{forF3qq}
  \sum_{k\in \Z}|\widehat{h}(k)|\lesssim \log \(s(q+1)\),
\end{equation}
where
$
\{\widehat{h}(k)\}
$
are the Fourier coefficients of $h$.

Now, using~\eqref{forF3qq}, we obtain
\begin{equation*}\label{qq2}
\begin{split}
&\left\Vert \sum_{\bm{k}^d=0}^{[\bm{\L}^d]}e^{i(\bm{k}^d,\bm{x}^d)}\{ \L_{d+1}(\bm{k}^{d}) \}^s \right\Vert_{L_1(\T^d)}
=\int_{\T^{d}}
\left|  \sum_{\bm{k}^d=0}^{[\bm{\L}^d]}e^{i(\bm{k}^d,\bm{x}^d)}   h\(\L_{d+1}(\bm{k}^{d})\)  \right| {\rm d}\bm{x}^{d}\\
&=\int_{\T^{d}}
\left| \sum_{\bm{k}^d=0}^{[\bm{\L}^d]}e^{i(\bm{k}^d,\bm{x}^d)} \(\sum_{\nu\in\Z}\widehat{h}(\nu)e^{2\pi i\nu
\L_{d+1}(\bm{k}^{d})   }  \) \right| {\rm d}\bm{x}^{d}\\
&\le \sum_{\nu\in\Z}|\widehat{h}(\nu)|\int_{\T^{d}}
\left| \sum_{\bm{k}^d=0}^{[\bm{\L}^d]}e^{i(\bm{k}^d,\,\bm{x}^d)} \right| {\rm d}\bm{x}^{d}\\
&\lesssim \log(s(q+1))\Vert D_{\bm{M}^{(d)}}\Vert_{L_1(\T^{d})}.
\end{split}
\end{equation*}

\end{proof}

\begin{lemma}\label{lenM}
  Let $N\ge \deg_l D_{\bm{M}^{(d+1)}}$, $l=1,\dots,d+1$. In terms of Lemma~\ref{lemod}, one has
\begin{equation}\label{nM}
 \left\Vert \sum_{\bm{k}^d=0}^{[\bm{\L}^d]}e^{i(\bm{k}^d,\,\bm{x}^d)}\{ \L_{d+1}(\bm{k}^{d}) \}^s \right\Vert_{L_1(\T^d)}
\lesssim (\log s(N+1)+2^s)\sum_{j=1}^R\Vert D_{\bm{M}_j^{(d)}}\Vert_{L_1(\T^{d})}.
\end{equation}
\end{lemma}

\begin{proof}
  By Dirichlet's theorem on simultaneous diophantine approximation, for any $Q>1$ there exist $p_l\in \Z_+$, $l=1,\dots, d+1$, and $q\in \N$, $1\le q\le Q$, such that
$$
\left|n_{d+1} -\frac{p_{d+1}}{q}\right|<\frac1{qQ^\frac1{d+1}}
$$
and
$$
\left|m_l^{(d)} -\frac{p_{l}}{q}\right|<\frac1{qQ^\frac1{d+1}},\quad l=1,\dots,d.
$$

Denote
$$
\g_{d+1}=n_{d+1} -\frac{p_{d+1}}{q},\quad \g_l=m_l^{(d)} -\frac{p_{l}}{q},\quad l=1,\dots,d,
$$
$$
\bm{\g}^l=(\g_1,\dots,\g_{l}),\quad \bm{p}^l=(p_1,\dots,p_{l}),\quad l=1,\dots,d+1.
$$

Let us take $Q=N^{(d+1)^2}$. In what follows, we may suppose that $N\ge 2$. Then it is easy to see that
\begin{equation*}
  \begin{split}
     \{\La_{d+1}(\bm{k}^{d})\}&=\g_{d+1}-(\bm{\g}^d,\bm{k}^d)+\left\{\widetilde{\La}_{d+1}(\bm{k}^{d})\right\}\\
&+\left\{
   \begin{array}{ll}
     1, & \hbox{$\g_{d+1}<(\bm{\g}^d,\bm{k}^d)$ and $\widetilde{\La}_{d+1}(\bm{k}^{d})\in\Z_+$,} \\
     0, & \hbox{otherwise,}
   \end{array}
 \right.
   \end{split}
\end{equation*}
where
$$
\widetilde{\La}_{d+1}(\bm{k}^{d})=\frac{p_{d+1}-(\bm{p}^d,\bm{k}^d)}{q}.
$$
Thus, we get
\begin{equation*}
  \begin{split}
    S&:= \sum_{\bm{k}^d=0}^{[\bm{\L}^d]}e^{i(\bm{k}^d,\,\bm{x}^d)}\{ \L_{d+1}(\bm{k}^{d}) \}^s
    =\sum_{\bm{k}^d=0}^{[\bm{\L}^d]}
e^{i(\bm{k}^d,\,\bm{x}^d)}
\(\g_{d+1}-(\bm{\g}^d,\bm{k}^d)+\left\{\widetilde{\La}_{d+1}(\bm{k}_1^{d})\right\}\)^s\\
&+\sum_{\underset{\g_{d+1}<(\bm{\g}^d,\bm{k}^d),\,\widetilde{\La}_{d+1}(\bm{k}_1^{d})\in\Z_+}{\bm{k}^d=0}}^{[\bm{\L}^d]}
e^{i(\bm{k}^d,\,\bm{x}^d)}:=S_1+S_2.
   \end{split}
\end{equation*}

Let us consider the polynomial $S_1$. We have
\begin{equation*}
  \begin{split}
     S_1&= \sum_{\bm{k}^d=0}^{[\bm{\L}^d]}e^{i(\bm{k}^d,\,\bm{x}^d)}\sum_{\nu=0}^s\binom{s}{\nu}
\(\g_{d+1}-(\bm{\g}^d,\bm{k}^d)\)^\nu\left\{\widetilde{\La}_{d+1}(\bm{k}^{d})\right\}^{s-\nu}
\\
&=\sum_{\bm{k}^d=0}^{[\bm{\L}^d]}e^{i(\bm{k}^d,\,\bm{x}^d)} \left\{\widetilde{\La}_{d+1}(\bm{k}^{d})\right\}^{s}
\\
&+\sum_{\nu=1}^s \binom{s}{\nu} \sum_{\bm{k}^d=0}^{[\bm{\L}^d]}e^{i(\bm{k}^d,\,\bm{x}^d)} \(\g_{d+1}-(\bm{\g}^d,\bm{k}^d)\)^\nu\left\{\widetilde{\La}_{d+1}(\bm{k}^{d})\right\}^{s-\nu}\\
&:=S_{11}+S_{12}.
   \end{split}
\end{equation*}
Taking into account that $q\le Q=N^{(d+1)^2}$ and using~\eqref{qq}, we obtain
\begin{equation}\label{S11nM}
  \Vert S_{11}\Vert_{L_1(\T^d)} \lesssim \log\(s(N^{(d+1)^2}+1)\)\Vert D_{\bm{M}^{(d)}}\Vert_{L_1(\T^{d})}.
\end{equation}
Since $|\g_{d+1}-(\bm{\g}^d,\bm{k}^d)|\lesssim \frac1{N^{d+1}}$, it follows that
\begin{equation}\label{S12nM}
  \Vert S_{12}\Vert_{L_1(\T^d)}\le \sum_{\nu=1}^s\binom{s}{\nu}\sum_{{\bm k}^d=0}^{[{\bm \L^d}]}\frac1{N^{d+1}}\lesssim 2^s\lesssim 2^s\Vert D_{\bm{M}^{(d)}}\Vert_{L_1(\T^{d})}.
\end{equation}
Thus, combining \eqref{S11nM} and~\eqref{S12nM}, we derive
\begin{equation}\label{eS1}
\begin{split}
   \Vert S_{1}\Vert_{L_1(\T^d)} \lesssim \(\log s(1+N)+ 2^s\)\Vert D_{\bm{M}^{(d)}}\Vert_{L_1(\T^{d})}.
\end{split}
\end{equation}

At the same time, by Lemma~\ref{lexs} and Lemma~\ref{lemod}, we obtain
\begin{equation}\label{thatS2}
  \Vert S_{2}\Vert_{L_1(\T^d)} \lesssim \sum_{j=1}^R \Vert
D_{\bm{{M}}_j^{(d)}}\Vert_{L_1(\T^{d})}.
\end{equation}

Finally, combining~\eqref{eS1} and~\eqref{thatS2}, we get~\eqref{nM}.

\end{proof}

\begin{lemma}\label{leFF}
Let $N\ge \deg_l D_{\bm{M}^{(d+1)}}$, $l=1,\dots,d+1$. In terms of Lemma~\ref{lemod}, one has
  \begin{equation*}
 \Vert F_{\bm{M}^{(d+1)}}\Vert_{L_1(\T^{d+1})}\lesssim \log(N+1) \sum_{j=1}^R \Vert
D_{\bm{{M}}_j^{(d)}}\Vert_{L_1(\T^d)}.
\end{equation*}
\end{lemma}

\begin{proof}
By Lemma~\ref{leforF}, Lemma~\ref{leD}, and Lemma~\ref{lenM}, we obtain
  \begin{equation*}
  \begin{split}
     \Vert F_{\bm{M}^{(d+1)}}\Vert_{L_1(\T^{d+1})}&\lesssim \sum_{s=1}^\infty \frac{(2\pi)^s}{s!}\Bigg( \sum_{j=1}^{\tilde{R}} \Vert
D_{\bm{\tilde{M}}_j^{(d)}}\Vert_{L_1(\T^d)}\\
&\quad\quad\quad+\sum_{\nu=1}^s\binom{s}{\nu}\bigg\Vert \sum_{\bm{k}^d=0}^{[\bm{\L}^d]}e^{i(\bm{k}^d,\bm{x}^d)}\{ \L_{d+1}(\bm{k}^{d}) \}^\nu \bigg\Vert_{L_1(\T^{d})}\Bigg)\\
&\lesssim \sum_{s=1}^\infty \frac{(2\pi)^s}{s!}
\Bigg(
   \sum_{j=1}^{\tilde{R}} \Vert
      D_{\bm{\tilde{M}}_j^{(d)}}\Vert_{L_1(\T^d)}\\
      &\quad\quad\quad+
        \sum_{\nu=1}^s\binom{s}{\nu}
             \bigg(
                (\log\nu(N+1)+2^\nu)
                \sum_{j=1}^{\bar{R}}\Vert D_{\bm{\bar{M}}_j^{(d)}}\Vert_{L_1(\T^{d})}
              \bigg)
\Bigg)\\
&\lesssim \log(N+1)\sum_{s=1}^\infty \frac{(2\pi)^s 4^s}{s!}\sum_{j=1}^{R}\Vert D_{\bm{{{M}}}_j^{(d)}}\Vert_{L_1(\T^{d})},
  \end{split}
\end{equation*}
where
$$
\sum_{j=1}^{R}\Vert D_{\bm{{{M}}}_j^{(d)}}\Vert_{L_1(\T^{d})}=\sum_{j=1}^{\tilde{R}} \Vert
      D_{\bm{\tilde{M}}_j^{(d)}}\Vert_{L_1(\T^d)}+\sum_{j=1}^{\bar{R}}\Vert D_{\bm{\bar{M}}_j^{(d)}}\Vert_{L_1(\T^{d})}.
$$
The lemma is proved.
\end{proof}

Now, let us find an estimate of the Lebesgue constant for the following Dirichlet kernel:
\begin{equation}\label{FFF}
  \begin{split}
    D_{ \bm{M}^{(d)}}(\bm{x}^d)
=\sum_{k_1=0}^{[ \L_1]}\sum_{k_2=0}^{[ \L_2(k_1)]}\dots \sum_{k_d=0}^{[\L_d(\bm{k}^{d-1})]}e^{i(\bm{k}^d,\,\bm{x}^d)}.
   \end{split}
\end{equation}

\begin{lemma}\label{thM} {\sc (Main Lemma)}
Let $N_l=\deg_l D_{\bm{M}^{(d)}}\ge 1$, $l=1,\dots,d$.
Then
\begin{equation}\label{M1}
  \left\Vert D_{\bm{M}^{(d)}}\right\Vert_{L_1(\T^d)}\lesssim \prod_{l=1}^d\log (N_l+1).
\end{equation}
\end{lemma}

\begin{proof}
One can suppose that $N_1\le N_2\le\dots\le N_d$. Otherwise, one can change the order of summation (using equality~\eqref{equ}) and rearrange the variables $(x_1,\dots,x_d)$ such that the following representation holds:
$$
D_{\bm{M}^{(d)}}(\bm{x^d})=\sum_{j=1}^{r} \e_j D_{\bm{M}_j^{(d)}}\(x_{i_1^{(j)}},\dots,x_{i_d^{(j)}}\),
$$
where $r\le C(d)$, $\e_j\in \{-1,1\}$, and $(i_1^{(j)},\dots,i_d^{(j)})$ is a rearrangement of $(1,\dots,d)$  such  that
\begin{equation}\label{zvezd1}
  \deg_{i_1^{(j)}} D_{\bm{M}_j^{(d)}}\le\dots\le\deg_{i_d^{(j)}} D_{\bm{M}_j^{(d)}},\quad j=1,\dots,r,
\end{equation}
and
\begin{equation}\label{zvezd2}
\deg_{i_l^{(j)}} D_{\bm{M}_j^{(d)}}\le \deg_{i_l^{(j)}} D_{\bm{M}^{(d)}},\quad l=1,\dots,d,\quad j=1,\dots,r.
\end{equation}
Therefore, since
$$
\Vert D_{\bm{M}^{(d)}}\Vert_{L_1(\T^d)}\le \sum_{j=1}^{r} \Vert D_{\bm{M}_j^{(d)}}\Vert_{L_1(\T^d)},
$$
to prove the lemma, one has to estimate the norm of $D_{\bm{M}_j^{(d)}}$ for each $j=1,\dots,r$ and take into account~\eqref{zvezd1} and~\eqref{zvezd2}.

Let us prove \eqref{M1} by induction.
For $d=1$, the inequality \eqref{M1} is obvious. Suppose that for any $s=2,\dots, d$ one has
\begin{equation}\label{M2}
  \left\Vert D_{\bm{\tilde{M}}^{(s)}}\right\Vert_{L_1(\T^s)}\lesssim \prod_{l=1}^s\log (N_l+1),
\end{equation}
where $D_{\bm{\tilde{M}}^{(s)}}$ is some polynomial of the form (\ref{FFF}) such that $\deg_l D_{\bm{\tilde{M}}^{(s)}}\le C(d)N_l$, $l=1,\dots,s$.
Let us prove that
\begin{equation}\label{M3}
  \left\Vert D_{\bm{M}^{(d+1)}}\right\Vert_{L_1(\T^{d+1})}\lesssim \prod_{l=1}^{d+1}\log (N_l+1),
\end{equation}
where
$$
D_{ \bm{M}^{(d+1)}}(\bm{x}^{d+1})=\sum_{k_1=0}^{[ \L_1 ]}\sum_{k_2=0}^{[ \L_2(k_1)]}\dots \sum_{k_d=0}^{[\L_d(\bm{k}^{d-1})]} \sum_{k_{d+1}=0}^{[\L_{d+1}(\bm{k}^{d})]} e^{i(\bm{k}^{d+1},\,\bm{x}^{d+1})},
$$
$N_{d+1}=\deg_{d+1}D_{ \bm{M}^{(d+1)}}\ge N_d$, and
$\L_{d+1}(\bm{k}^{d})\ge 0$ for all admissible vectors $\bm{k}^d$.

Indeed, by Lemma~\ref{leFG}, Lemma~\ref{leG}, and Lemma~\ref{leFF}, we obtain
\begin{equation}\label{zvezd3}
  \begin{split}
     \left\Vert D_{\bm{M}^{(d+1)}}\right\Vert_{L_1(\T^{d+1})}&\le \Vert G_{\bm{M}^{(d+1)}} \Vert_{L_1(\T^{d+1})} +\Vert F_{\bm{M}^{(d+1)}} \Vert_{L_1(\T^{d+1})}\\
&\lesssim \log (N_{d+1}+1)\Vert D_{\bm{M}^{(d)}} \Vert_{L_1(\T^{d})}+\log (N_{d+1}+1) \sum_{j=1}^R \Vert
D_{\bm{{M}}_j^{(d)}}\Vert_{L_1(\T^d)},
   \end{split}
\end{equation}
where $D_{\bm{{M}}_j^{(d)}}$ is such that $\deg_l D_{\bm{{M}}_j^{(d)}}\le C(d)N_l$ for all $l=1,\dots,d$ and $j=1,\dots,R$. Therefore, applying~\eqref{M2} to the last inequality in~\eqref{zvezd3}, we get~\eqref{M3}.

The lemma is proved.
\end{proof}

\section{Main results}

In the following theorem, we obtain an estimate from above of the Lebesgue constant for a general convex polyhedron.

\begin{theorem}\label{cor1} {\sc (Main Theorem)}
  Let $P\subset \R^d$ be a bounded convex polyhedron such that $P\subset [0,n_1]\times\dots\times [0,n_d]$, $n_j\ge 1$, $j=1,\dots,d$, and let $s$ be size of the triangulation of~$P$. Then
  \begin{equation}\label{T1}
    \mathcal{L}(P)\le C(s,d)\prod_{j=1}^d\log (n_j+1).
  \end{equation}
  Moreover, if $\min_{j=1,\dots, d} n_j\to \infty$ in~\eqref{T1}, then $C(s,d)\le C(d)s$.
\end{theorem}

\begin{proof}
We start from the triangulation of the polyhedron. Let $P$ be triangulated by $s$ tetrahedra $T_j$ such that
$$
P=\bigcup_{j=1}^{s}T_j
$$
and $T_j\cap T_i$, $i\neq j$, is either empty or a face of both tetrahedra (see~\cite{BEG},~\cite[p. 842]{RMGGS}). Using the inclusion–exclusion principle, we obtain
\begin{equation}\label{in-ec}
  \mathcal{L}(P)\le \sum_{j=1}^s \mathcal{L}(T_j)+\sum_{\nu=2}^s\(\sum_{1\le l_1<\dots<l_\nu\le s}\mathcal{L}(T_{l_1}\cap\dots\cap T_{l_\nu})\).
\end{equation}

Note that the dimension of the tetrahedron $T_{l_1}\cap\dots\cap T_{l_\nu}$ is less than $d$.
Thus, to prove the theorem it is sufficient to prove \eqref{T1} for any tetrahedron $T$ such that $T\subset [0,n_1]\times\dots\times [0,n_d]$. In particular, this and~\eqref{in-ec}  yield
$$
\mathcal{L}(P)\le C(d)s\prod_{j=1}^d \log(n_j+1)+C(d,s)\sum_{j=1}^d\prod_{\underset{i\neq j}{i=1}}^d\log(n_j+1).
$$
This inequality implies the statement of the theorem.

Suppose that a tetrahedron $T\subset [0,n_1]\times\dots\times [0,n_d]$ is given as a set of vectors ${\bm \xi}^d\in\R_+^d$ such that
\begin{equation}\label{T2}
  (\bm{\a}_l^{(d)},\bm{\xi}^d)\le \b_l,
\end{equation}
where $\bm{\a}_l^{(d)}=(\a_{l1}^{(d)},\dots,\a_{ld}^{(d)})\in \R^d$ and $\b_l\in\R$, $l=1,\dots, d+1$.

Solving the system of inequalities~\eqref{T2} by the Fourier–Motzkin elimination method (see, e.g.,~\cite[Ch. 12]{Sc} and~\cite{Sch}), one can verify that  $T$ can be represented in the following form
$$
T=\bigcup_{j=1}^{R}\mathcal{P}_j,
$$
where $R\le C(d)$ and $\mathcal{P}_j\subset T$ are (non-closed) polyhedra such that $\mathcal{P}_j\cap \mathcal{P}_i=\varnothing$, $i\ne j$, and  for each $j=1,\dots,R$ the set $\mathcal{P}_j$ can be defined as a set of vectors $\bm{\xi}^d\in\R_+^d$ satisfying the system
\begin{equation*}
\left\{
  \begin{array}{ll}
    \underline{\L}_{j1}  \preccurlyeq \xi_1 \preccurlyeq \overline{\L}_{j1}, \\
    \underline{\L}_{js}(\bm{\xi}^{s-1})  \preccurlyeq  \xi_s \preccurlyeq  \overline{\L}_{js}(\bm{\xi}^{s-1}), & \hbox{$s=2,\dots,d,$}
  \end{array}
\right.
\end{equation*}
where $\underline{\L}_{j1}$ and $\overline{\L}_{j1}$ are positive numbers and the functions $\underline{\L}_{js}$ and $\overline{\L}_{js}$, $s=2,\dots,d$,  have the form~\eqref{FORML} (we associate $\bm{\overline{\L}}_{j}^d=(\overline{\L}_{j1},\dots,\overline{\L}_{jd})$ and $\bm{\underline{\L}}_{j}^d=(\underline{\L}_{j1},\dots,\underline{\L}_{jd})$
with the matrixes $\bm{\overline{M}}_j^{(d)}$ and $\bm{\underline{M}}_j^{(d)}$, correspondingly).
Thus, we have
\begin{equation}\label{T4}
  \Bigg\Vert \sum_{\bm{k}\in T\cap \Z_+^d}e^{i(\bm{k},\,\bm{x})}\bigg\Vert_{L_1(\T^d)}\le \sum_{j=1}^{R}\Bigg\Vert \sum_{\bm{k}\in \mathcal{P}_j\cap \Z_+^d}e^{i(\bm{k},\,\bm{x})}\Bigg\Vert_{L_1(\T^d)}.
\end{equation}
By using equality~\eqref{equ}, we derive the following representation for each $\mathcal{P}_j$, $j=1,\dots,R$,
\begin{equation}\label{=}
  \sum_{\bm{k}\in \mathcal{P}_j\cap \Z_+^d}e^{i(\bm{k},\,\bm{x})}=\sum_{\nu=1}^{\tilde{R}_j}\e_{\nu,j} D_{\bm{\tilde{M}}_{\nu,j}^{(d)}}(\bm x^d),
\end{equation}
where $\tilde{R}_j\le {C}(d)$, $\e_{\nu,j}\in\{-1,0,1\}$, and $D_{\bm{\tilde{M}}_{\nu, j}^{(d)}}$ has the form~\eqref{FFF}.
It is obvious that we can choose the matrixes $\bm{\tilde{M}}_{\nu,j}^{(d)}$ such that $\deg_l D_{\bm{\tilde{M}}_{\nu,j}^{(d)}}\le n_l$ for each $l=1,\dots,d$, $j=1,\dots,R$, and $\nu=1,\dots,\tilde{R}_j$. Finally, combining~\eqref{T4} and~\eqref{=} and applying Lemma~\ref{thM} to each $D_{\bm{\tilde{M}}_{\nu,j}^{(d)}}$, we derive
\begin{equation*}
  \begin{split}
     \mathcal{L}(T)\le \sum_{j=1}^{R}\sum_{\nu=1}^{\tilde{R}_j}\Vert D_{\bm{\tilde{M}}_{\nu,j}^{(d)}} \Vert_{L_1(\T^d)}\le C(d)\prod_{j=1}^d\log(n_j+1).
   \end{split}
\end{equation*}

The theorem is proved.
\end{proof}

\begin{remark}
In the case $d=2$, more accurate calculations show that the inequality~\eqref{T1} holds with $C(s,2)=cs$, where $c$ is some absolute constant.
\end{remark}

In the next theorem, we obtain an estimate from below of the Lebesgue constant for one class of convex polyhedra. In particular, the result below shows the sharpness of Theorem~\ref{cor1}.

\begin{theorem}\label{thBel}
  Let $P$ be a bounded convex polyhedron in $\R^d$ such that $[0,n_1]\times\dots\times [0,n_d]\subset P_m\subset \R_+^d$ and let $n_j\ge 1$, $j=1,\dots,d$. Then
  \begin{equation}\label{eqBel}
    \prod_{j=1}^d\log (n_j+1)\lesssim\mathcal{L}(P).
  \end{equation}
\end{theorem}

\begin{proof}
By using a multidimensional generalization of~Hardy's inequality (see~\cite[p.~69]{Ru})
\begin{equation}\label{Hardy1}
  \sum_{k_1=0}^{N_1}\dots \sum_{k_d=0}^{N_d}\frac{|a_{\bm{k}}|}{(k_1+1)\dots (k_d+1)}
\lesssim \bigg\Vert \sum_{k_1=0}^{N_1}\dots \sum_{k_d=0}^{N_d} a_{\bm{k}} e^{i(\bm{k},\,\bm{x})}\bigg\Vert_{L_1(\T^d)}
\end{equation}
and the induction argument, we get
\begin{equation*}
\begin{split}
    \bigg\Vert \sum_{\bm{k}\in P\cap \Z_+^d} e^{i(\bm{k},\,\bm{x})}\bigg\Vert_{L_1(\T^d)}&\gtrsim \sum_{\bm{k}\in P\cap \Z_+^d}\frac1{(k_1+1)(k_2+1)\dots(k_d+1)} \\
    &\gtrsim \sum_{k_1=0}^{\lfloor n_1 \rfloor}\sum_{k_2=0}^{\lfloor n_2\rfloor}\dots \sum_{k_d=0}^{\lfloor n_d\rfloor} \frac1{(k_1+1)(k_2+1)\dots(k_d+1)}\\
    &\gtrsim \prod_{j=1}^d\log (n_j+1).
\end{split}
\end{equation*}
The theorem is proved.
\end{proof}

Now, let us consider some examples of application of Theorem~\ref{cor1} and Theorem~\ref{thBel}.

The simplest example is the rectangle $R_{\bm n}=[0,n_1]\times\dots\times[0,n_d]$. One has (see~\eqref{Rectan+})
$$
\mathcal{L}(R_{\bm n})\asymp \prod_{j=1}^d\log(n_j+1).
$$

The problem becomes non-trivial for tetrahedra.

\begin{corollary}\label{corM}
Let $n_j\ge 1$, $j=1,\dots, d$, and
$$
\D_{\bm{n}}=\left\{\bm{\xi}\in \R_+^d\,:\,\sum_{j=1}^d\frac{\xi_j}{n_j}\le 1\right\}.
$$
Then
\begin{equation*}\label{cM1}
  \mathcal{L}(\D_{\bm n})\asymp \prod_{j=1}^d\log (n_j+1).
\end{equation*}
\end{corollary}

\begin{proof}
To prove the corollary, it is sufficient to note that
$$
[0,n_1/d]\times\dots\times[0,n_d/d]\subset \D_{\bm n}\subset [0,n_1]\times\dots\times[0,n_d]
$$
and use Theorem~\ref{cor1} and Theorem~\ref{thBel}.
\end{proof}

By analogy, we can prove the following result which can be applied  in multivariate polynomial interpolation on the Lissajous-Chebyshev nodes (see~\cite{DE}, see also~\cite{E} in the case $d=2$).

\begin{corollary}\label{thMQ}
Let $n_j\ge 1$, $j=1,\dots, d$, and
$$
T_{\bm{n}}=\left\{\bm{\xi}\in\R_+^{d}:\frac{\xi_i}{n_i}+\frac{\xi_j}{n_j}\leq 1,\quad \textrm{for}\quad i\neq j,\quad\,i,\,j=1,\dots,d\right\}.
$$
Then
\begin{equation}\label{M1QT}
  \mathcal{L}(T_{\bm n})\asymp \prod_{j=1}^d\log (n_j+1).
\end{equation}
\end{corollary}

\section{Estimates of the $L_p$ Lebesgue constant for convex polyhedra}\label{SecLp}

Let
$$
\mathcal{L}(W)_p:=\(\frac1{(2\pi)^d}\int_{\T^d}\bigg|\sum_{\bm{k}\in W\cap\Z^d} e^{i(\bm{k},\,\bm{x})}\bigg|^p {\rm d}{\bm x}\)^\frac1p
$$
be the $L_p$ Lebesgue constant for the set $W\subset\R^d$.

Above, we obtained the estimates of $\mathcal{L}(W)_p$ for convex polyhedra in the case $p=1$. It turns out that all these results can be transferred to the case $1<p<\infty$ after some minor changes.

In particular, in the following results, we improve and generalize the main results in~\cite{Ash} and~\cite{AC}. Everywhere below, $1<p<\infty$ and constants in "$\lesssim$" and "$\gtrsim$" depend only on $p$ and $d$.

\begin{theorem}\label{thMp}
  Let $P\subset \R^d$ be a bounded convex polyhedron such that $P\subset [0,n_1]\times\dots\times [0,n_d]$, $n_j\ge 1$, $j=1,\dots,d$, and let $s$ be size of the triangulation of~$P$. Then
  \begin{equation}\label{T1p}
    \mathcal{L}(P)\le C(s,d,p)\prod_{j=1}^d(n_j+1)^{1-\frac1p}.
  \end{equation}
  Moreover, if $\min_{j=1,\dots, d} n_j\to \infty$ in~\eqref{T1p}, then $C(s,d,p)\le C(d,p)s$.
\end{theorem}

\begin{proof}
Inequality~\eqref{T1p} can be proved repeating step by step the proof of main Lemma~\ref{thM} and other auxiliary lemmas. Here, we only note that instead of~(\ref{Slog}) one has to use the inequality
\begin{equation*}\label{Slogp}
  \Vert S_{t}\Vert_{L_p(\T^1)}\lesssim (t+1)^{1-\frac1p}.
\end{equation*}
\end{proof}

\begin{theorem}\label{corMp}
  Let $P$ be a bounded convex polyhedron in $\R^d$ such that $[0,n_1]\times\dots\times [0,n_d]\subset P\subset \R_+^d$ and let $n_j\ge 1$, $j=1,\dots,d$. Then
  \begin{equation}\label{eqBelp}
    \prod_{j=1}^d(n_j+1)^{1-\frac1p}\lesssim\mathcal{L}(P)_p.
  \end{equation}
\end{theorem}

\begin{proof}
The proof of estimate~\eqref{eqBelp} is almost the same as the proof of~\eqref{eqBel}. The only difference is that  instead of~\eqref{Hardy1} we have to use the following
$L_p$ Hardy-Littlewood inequality (see~\cite{D})
$$
\left(\sum_{k_1=0}^{N_1}\dots \sum_{k_d=0}^{N_d}\frac{|a_{\bm{k}}|^p}{((k_1+1)\dots (k_d+1))^{2-p}} \right)^\frac1p
\lesssim \bigg\Vert \sum_{k_1=0}^{N_1}\dots \sum_{k_d=0}^{N_d} a_{\bm{k}} e^{i(\bm{k},\,\bm{x})}\bigg\Vert_{L_p(\T^d)},
$$
where the coefficients $\{a_{\bm{k}}\}_{\bm{k}\in\Z_+^d}$ satisfy the condition $a_{\bm{k}}\le a_{\bm{m}}$ if $k_l\ge m_l$ for all $l=1,\dots,d$.
\end{proof}


\begin{thebibliography}{16}
{

\bibitem{Ash} M. Ash, Triangular Dirichlet Kernels and Growth of $L^p$ Lebesgue Constants, J. Fourier Anal. Appl.  \textbf{16}, no. 6 (2010), 1053–1069.

\bibitem{AC} J.M. Ash, L. De Carli, Growth of $L^p$ Lebesgue constants for convex polyhedra and other regions.
Trans. Am. Math. Soc. \textbf{361} (2009), 4215–4232.

\bibitem{Ba} S.P. Baiborodov, Lebesgue Constants of Polyhedra, Mat. Zametki \textbf{32} (1982), 817–
822 (Russian); English translation in Math. Notes \textbf{32} (1982), 895–898.

\bibitem{Be} E.S. Belinsky, Behavior of the Lebesgue constants of certain methods of summation of
multiple Fourier series, Metric Questions of the Theory of Functions and Mappings,
Naukova Dumka, Kiev, 1977, 19–39 (Russian).

\bibitem{BEG} M. Bern, D. Eppstein, J. Gilbert, Provably good mesh generation, J. Comput. Syst. Sci. \textbf{48} (1994), 384-409.

\bibitem{DE} P. Dencker, W. Erb, Multivariate polynomial interpolation on Lissajous-Chebyshev nodes, arXiv:1511.04564v1 [math.NA] 14 Nov 2015.

\bibitem{DeLo} {R.A. DeVore, G.G. Lorentz},  Constructive Approximation, Springer-Verlag, New York, 1993.


\bibitem{D} M.I. Dyachenko,  Multiple trigonometric series with lexicographically monotone coefficients, Anal. Math.
 \textbf{16}, no.~3 (1990), 173-190.

\bibitem{D2} M.I. Dyachenko, Some problems in the theory of multiple trigonometric series, Uspekhi Mat.
Nauk \textbf{47}, no.~5 (1992), 97–162 (Russian); English translation in Russian Math. Surveys
\textbf{47}, no.~5 (1992), 103–171.


\bibitem{E} W. Erb, Bivariate Lagrange interpolation at the node points of Lissajous curves - the degenerate case, Appl. Math. Comput. \textbf{289} (2016), 409--425.


\bibitem{Ku} O.I. Kuznetsova, The Asymptotic Behavior of the Lebesgue Constants for a Sequence of
Triangular Partial Sums of Double Fourier Series, Sib. Mat. Zh. \textbf{XVIII} (1977),
629–636 (Russian); English translation in Siberian Math. J. \textbf{18} (1977), 449–454.

\bibitem{L} E.R. Liflyand, Lebesgue Constants of multiple Fourier series, Online J. Anal.
Comb. \textbf{1}, no. 5 (2006), 1-112.

\bibitem{NP} F. Nazarov, A. Podkorytov, On the behavior of the Lebesgue constants for two dimensional
Fourier sums over polygons, Algebra i Analiz \textbf{7} (1995), 214–238 (Russian); English translation in St.-Petersburg Math.~J. \textbf{7} (1995), 663–680.

\bibitem{P} A.N. Podkorytov, The order of growth of the Lebesgue constants of Fourier sums over polyhedra,
Vestnik Leningrad. Univ. Matem. \textbf{7} (1982), 110–111 (Russian).

\bibitem{P2} A.N. Podkorytov, On asymptotics of Dirichlet's kernels of Fourier sums with respect to a polygon, Investigations on linear operators and function theory. Part XV, Zap. Nauchn. Sem. LOMI \textbf{149}, Nauka, Leningrad. Otdel., Leningrad (1986), 142–149.

\bibitem{RMGGS} K.H. Rosen, J.G. Michaels, J.L. Gross, J.W. Grossman, D.R. Shier, Handbook of Discrete and Combinatorial Mathematics. Boca. Raton, FL: CRC Press, 2000.

\bibitem{Ru} W. Rudin, Function Theory in Polydiscs, Benjamin, New York, 1969.

\bibitem{Sch}  M. Schechter, Integration over a polyhedron: an application of the Fourier-Motzkin elimination method, Amer. Math. Monthly \textbf{105}, no. 3 (1998), 246–251.

\bibitem{Sc}  A. Schrijver,
Theory of linear and integer programming, Wiley, Chichester, 1986.

\bibitem{Sk} M.A. Skopina, Lebesgue constants of multiple polyhedron sums of de la Vallee–Poussin, Analytical theory of numbers and theory of functions. Part 5, Zap. Nauchn. Sem. LOMI \textbf{125}, Nauka, Leningrad. Otdel., Leningrad, 1983, 154–165.

\bibitem{TB} R.M. Trigub, E.S. Belinsky, { Fourier Analysis and
Appoximation of Functions}, Kluwer, 2004.

\bibitem{YuYu} A.A. Yudin, V.A. Yudin, Polygonal Dirichlet kernels and growth of Lebesgue constants, Mat. Zametki
\textbf{37} (1985), 220–236  (Russian); English translation in Math. Notes \textbf{37} (1985),  124–135.



}

\end{thebibliography}
\end{document}